\documentclass[preprint]{elsarticle}

\usepackage{ucs} 
\usepackage[utf8x]{inputenx}
\usepackage[T1]{fontenc}
\usepackage{lmodern}

\usepackage[english]{babel}
\usepackage{soul}
\usepackage{verbatim}

\usepackage{amsmath, amssymb}
\usepackage{amsthm} 
\usepackage{amsfonts}
\usepackage{textcomp}
\usepackage{mathcomp} 

\usepackage{graphicx}
\usepackage[center]{caption}
\usepackage{subcaption}
\usepackage{floatrow} 
\usepackage{booktabs}
\usepackage{epic,eepic}
\usepackage{color}
\usepackage[table]{xcolor}
\usepackage{tikz}

\usepackage{hyperref}
\usepackage{hypcap}
\usepackage{url}
\usepackage{fullpage}

\usepackage{enumerate}
\usepackage{paralist}
\usepackage[final]{showlabels} 

\theoremstyle{plain}
\newtheorem{thm}{Theorem}
\newtheorem{lem}[thm]{Lemma}
\newtheorem{cor}[thm]{Corollary}
\newtheorem{prop}[thm]{Proposition}

\newtheorem{obs}[thm]{Observation}

\newtheorem*{claim-non}{Claim}
\theoremstyle{definition}

\newtheorem{conj}[thm]{Conjecture}

\newtheorem{ques}[thm]{Question}
\theoremstyle{remark}

\definecolor{gray}{rgb}{.4,.4,.4}

\newcommand{\defi}[1]{\textit{#1}}

\def\ignore#1{{}}

\makeatletter
\def\ps@pprintTitle{%
 \let\@oddhead\@empty
 \let\@evenhead\@empty
 \def\@oddfoot{\footnotesize\itshape\hfill\today}%
 \let\@evenfoot\@oddfoot}
\makeatother

\begin{document}

\begin{frontmatter}

\title{Spanning trees with nonseparating paths\tnoteref{todos}}
\tnotetext[todos]{Research partially supported by CNPq (Proc.~477203/2012-4), FAPESP (Proc.~2013/03447-6), and Project MaCLinC of NUMEC/USP.}

\author[ccc]{Cristina~G.~Fernandes\fnref{thkcris}}
\ead{cris@ime.usp.br}

\author[ccc]{C\'esar~Hern\'andez--V\'elez\corref{cor1}\fnref{thkcesar}}
\ead{israel@ime.usp.br}

\author[lee]{Orlando~Lee\fnref{thklee}}
\ead{lee@ic.unicamp.br}

\author[ccc]{Jos\'e~C.~de~Pina}
\ead{coelho@ime.usp.br}

\cortext[cor1]{Corresponding author}
\fntext[thkcris]{Partially supported by CNPq (Proc.~308523/2012-1).}
\fntext[thkcesar]{Supported by FAPESP (Proc.~2012/24597-3).}
\fntext[thklee]{Supported by Bolsa de Produtividade do CNPq (Proc.~303947/2008-0) and Edital Universal CNPq (Proc.~477692/2012-5).}


\address[ccc]{Instituto de Matem\'atica e Estat\'istica, Universidade de S\~ao Paulo. S\~ao Paulo -- SP, Brazil 05508-090}
\address[lee]{Instituto de Computa\c{c}\~ao, Universidade Estadual de Campinas. Campinas -- SP, Brazil 13083-852}

\begin{abstract}
We consider questions related to the existence 
of spanning trees in graphs with the property 
that after the removal of any path in the tree 
the graph remains connected.
We show that, for planar graphs, 
the existence of trees with this property is closely 
related to the Hamiltonicity of the graph.
For graphs with a 1- or 2-vertex cut,
the Hamiltonicity also plays a central role.
We also deal with spanning trees satisfying  
this property  restricted to
paths arising from 
fundamental cycles.
The cycle space of a graph can be generated by  
the fundamental cycles of any spanning tree,
and Tutte showed, that for a 3-connected graph, it 
can be generated by nonseparating cycles.
We are also interested in the existence of a  
fundamental basis consisting of 
nonseparating cycles.


\end{abstract}


\end{frontmatter}



\section{Introduction}
\label{sec:intro}

In this paper every graph is finite, simple and connected.
The \defi{vertex set} of a graph $G$ is denoted by $V(G)$ and its
\defi{edge set} by $E(G)$. 
For any subgraph $H$ of $G$, $G[H]$ is the subgraph of $G$ induced 
by $V(H)$, and $G - H$ is the subgraph of $G$ induced by 
$V(G) \setminus V(H)$.
We say $G$ is {\em $k$-connected} if $|V(G)| > k$ and, for every 
set $X \subseteq V(G)$ with $|X| < k$, $G - X$ is connected.
By convention, both the \defi{null graph} (the graph with no vertices, and hence no edges)
and the \defi{trivial graph} (any graph with just one vertex) are $0$-connected and $1$-connected,
but are not $k$-connected for any $k > 1$.

If $G$ is a graph and $P$ a path in $G$, we say that $P$ is a 
\defi{nonseparating (separating) path} if $G-P$ is connected
(respectively, disconnected). 
A path $P$ between $u$ and $v$ is called a {\em $uv$-path}.

Tutte~\cite{Tu63} proved that, for every $3$-connected graph $G$ and 
vertices $u$ and $v$, there exists a nonseparating $uv$-path.
In 1975, Lov\'asz~\cite{Lo75} made the following conjecture which is 
related to this result of Tutte.

\begin{conj}\label{conj:lovasz}
For every positive integer $k$, 
there exists a positive integer $f(k)$ such that, 
for every $f(k)$-connected graph $G$ and vertices $u$
and $v$, there exists a $uv$-path $P$ such that $G-P$ 
is $k$-connected.
\end{conj}

It is easy to see that $f(1) \geq 3$.
So, Tutte's result implies that $f(1) = 3$. 
Chen, Gould, and Yu~\cite{ChGo03}, 
and independently Kriesell~\cite{Kr01}, proved
that $f(2) = 5$. Recently, Kawarabayashi, Lee, and Yu~\cite{KaLe05}
proved that $f(2)=4$ except for double wheels. 
The conjecture is open for $k \geq 3$.

Some related questions have been settled yielding new conjectures. 
One of them is the following due to 
Kawarabayashi and Ozeki~\cite{KaOz11}.

\begin{conj}\label{conj:kawa}
For all positive integers $k$ and $\ell$, 
there exists a positive integer $g(k,\ell)$ such that 
the following holds. 
For every $g(k,\ell)$-connected graph $G$ and vertices $u$ and $v$, 
there exist internally disjoint $uv$-paths $P_1, \ldots, P_\ell$
such that $G - \bigcup_{i=1}^\ell P_i$ is $k$-connected.
\end{conj}

For  $\ell=1$  this conjecture  corresponds  to  Lovász  Conjecture.
Kawarabayashi and  Ozeki showed that  $g(1,\ell)=2\ell+1$ and $g(2,\ell)
\leq 3\ell+2$. 

Other related question was raised by Hong and Lai~\cite{HoLa13}, 
who considered the problem of connecting a subset of vertices  
by a tree instead of just two  vertices as in Tutte's theorem.
They made the conjecture below.

\begin{conj}\label{conj:hong}
For all positive integers $k$ and $r$, 
there exists a positive integer $h(k,r)$ such that,
for every $h(k,r)$-connected graph $G$ and 
subset $X$ with $r$ vertices,
there exists a tree $T$ connecting $X$ 
such that $G-T$ is $k$-connected.
\end{conj}

Hong and Lai proved that $h(1,r) = r + 1$ and $h(2,r)\leq 2r +1$.

By Tutte, in a $3$-connected graph, every pair of vertices is connected 
by a nonseparating path.
\textit{Is it possible to get a spanning tree where all paths
 are nonseparating?} 
Inspired by Tutte's result on nonseparating paths, 
we call \defi{Tutte tree} any spanning tree of a graph 
such that every path in the tree is nonseparating.
In this paper we deal with the following question.

\begin{ques}\label{ques:spantree}
When does a graph have a Tutte tree?
\end{ques}

For planar graphs we prove the following.
\begin{thm}\label{thm:tuttetree}
A planar graph  has a Tutte tree if and only if it is
Hamiltonian or it has a spanning tree whose leaves induce a triangle.
\end{thm}

We denote the set of leaves in a tree $T$ by $L(T)$. 
We prove the theorem below.

\begin{thm}\label{thm:3contt}
Let $G$ be a graph and $T$ be a spanning tree. 
If $G[L(T)]$ is $3$-connected, then  $T$ is a Tutte tree.
\end{thm}

A cycle $C$ is a \defi{nonseparating (separating) cycle}
if $G-C$ is connected (respectively, disconnected). 
We adopt the convention that a Hamiltonian cycle is nonseparating.
Let~$T$ be a spanning tree of $G$. 
For every $e \in E(G) \setminus E(T)$, there is a unique cycle $C_e$
in~$T+e$.
These cycles $C_e$ are called \defi{fundamental cycles} (of $G$) with 
respect to $T$. 

Tutte~\cite{Tu63} proved that the cycle space of a 3-connected graph is
generated by its nonseparating induced cycles.
Once more, inspired by a result of Tutte we say that a spanning tree $T$ 
of a graph is a \defi{fundamental Tutte tree} if any fundamental cycle 
with respect to $T$ is nonseparating.

We present a question concerning fundamental Tutte trees similar
to Question~\ref{ques:spantree}.

\begin{ques}\label{ques:fundspantree}
When does a graph have a fundamental Tutte tree?
\end{ques}

This question corresponds to asking if there is a fundamental basis of
the cycle space consisting of nonseparating cycles. For the case of
planar graphs, we prove the following.

\begin{thm}\label{thm:fundtuttetree}
Let $G$ be a planar graph and $T$ be a spanning tree.
If $T$ is a fundamental Tutte tree, then in any plane drawing of $G$ 
all leaves are in the same face.
\end{thm}

It is easy to see that if $T$ is a Tutte tree of $G$, then $T$ is a
fundamental Tutte tree.  However, as we shall see, not every graph
with a fundamental Tutte tree has a Tutte tree.

This paper is structured as follows. 
In Section~\ref{sec:tuttetree}, we dive into Tutte trees.  
We analyze Tutte trees in planar graphs and prove 
Theorem~\ref{thm:tuttetree}.
The structure of graphs with 2-vertex cut having a Tutte
tree is investigated.  
It is shown that the problem of deciding whether a 3-connected
graph has a Tutte tree is NP-complete and some examples of graphs
with no Tutte tree are presented.
Theorem~\ref{thm:3contt} is proved and examples showing 
that the sufficient condition of this theorem is not a necessary 
one are exhibited.
Next, in Section~\ref{sec:fundtuttetree}, the fundamental Tutte 
trees are contemplated.
The structure of graphs having a fundamental Tutte tree  
with a 1-vertex cut or a 2-vertex cut are explored. 
Theorem~\ref{thm:fundtuttetree} is proved and an example of a graph 
with a fundamental Tutte tree and no Tutte tree is exhibited.
We conclude this section presenting a graph with 
no fundamental Tutte tree.
Finally, in Section~\ref{sec:cr}, we present some concluding 
remarks and open questions.


\section{Tutte trees}
\label{sec:tuttetree}

Note that $3$-connectedness is a sufficient condition for a graph to
have a nonseparating path between any two vertices. However, it is
not a necessary condition. A cycle of length at least three, which is
$2$-connected but not $3$-connected, has a nonseparating path between
any two vertices. The following observation is trivial, but important.

\begin{obs}\label{obs:tutteconect}
Every graph $G$ with a Tutte tree has connectivity at least $2$.
\end{obs}

In order to prove Theorem~\ref{thm:tuttetree} we first prove a lemma.

\begin{lem}\label{lem:tutte}
Let $G$ be a planar graph and $T$ be a spanning tree.
If $T$ is a Tutte tree then $G[L(T)]$ is a clique.
\end{lem}

\begin{proof}
Suppose that $T$ is Tutte tree of $G$ and there exist two nonadjacent
leaves of $T$. We claim that there is a separating path in $T$.

Let $D$ be a plane drawing of $G$. 
Let $u$ and $v$ be two nonadjacent leaves of $T$. 
Since $G$ is $2$-connected, 
$u$ has at least two neighbors. 
Let $N(u)$ denote the set of neighbors of $u$ in $G$. 
For any two vertices $x, y \in N(u)$ there exists a unique 
path between them in~$T$, denoted by $xTy$ . 
Thus, $xTy \cup \{yu, ux\}$ is a cycle in $G$ denoted by $C_{xy}$, 
and it is represented by the Jordan curve $D[C_{xy}]$ in $D$.
As $v$ is not adjacent to $u$, we have that $x, y \neq v$, 
and as $v$ is a leaf, the cycle $C_{xy}$ does not contain $v$.
Without loss of generality, we may assume that $v$ is
inside the closed disk bounded by $D[C_{xy}]$.
Let $u_1, u_2 \in N(u)$ be such that: 
\begin{inparaenum}[(i)]
\item $v$ is inside the closed disk bounded by $D[C_{u_1u_2}]$, and  
\item if $C_{xy}$ is another cycle containing $v$ inside the 
      closed disk bounded by $D[C_{xy}]$ 
      then $D[C_{u_1u_2}]$ is contained in the closed disk bounded by $D[C_{xy}]$.
\end{inparaenum}

No vertex $u'\in N(u)$ is inside the closed
disk bounded by $D[C_{u_1u_2}]$. Indeed, suppose that there exists such a
$u'$. Since $D$ is a plane drawing of $G$, both, the path $D[u'Tu_1]$
and the path $D[u'Tu_2]$ are inside the closed disk bounded by
$D[C_{u_1u_2}]$. Thus one of $C_{u_1u'}$ or $C_{u'u_2}$ contains $v$ in
the interior of the Jordan curve that it induces, contradicting the
choice of~$C_{u_1u_2}$.

Also, no vertex in the interior of $D[C_{u_1u_2}]$ has as a neighbor a
vertex in the exterior. Therefore $G-u_1Tu_2$ has at least two different
components, one containing $u$ and another containing $v$. Therefore the
path $u_1Tu_2$ is separating.
\end{proof}

\begin{proof}[Proof of Theorem~\ref{thm:tuttetree}]
$(\Rightarrow)$ Let $T$ be a Tutte tree of a planar graph $G$.
By Lemma~\ref{lem:tutte}, $G[L(T)]$ is a clique. 
As $G$ is planar, $G[L(T)]$ is either an edge $e$ or a triangle, 
otherwise $G$ contains a $K_5$ minor.
In the first case, $T+e$ is a Hamiltonian cycle.

\noindent 
$(\Leftarrow)$ 
If $G$ has a Hamiltonian cycle $C$ and $e$ is an edge of $C$, then
$C-e$ is a Tutte tree. If $T$ is a spanning tree such that its leaves 
induce a triangle, then it is trivial to check that all paths in $T$ are nonseparating.
\end{proof}

\begin{cor}\label{cor:tuttehamil}
Let $G$ be a planar graph. If $G$ has a Tutte tree, then $G$ is \defi{traceable}, i.e., has a Hamiltonian path.
\end{cor}

\begin{proof}
Let $T$ be a Tutte tree in $G$. By Theorem~\ref{thm:tuttetree}, the
tree $T$ cannot have more than three leaves.  If $T$ has only two
leaves, then $G$ is Hamiltonian, and hence traceable. If $T$ has three
leaves, then $T$ is homeomorphic to a $3$-star, and it is not hard to
find a Hamiltonian path in $G$ using one of the edges between two of the
leaves of~$T$.
\end{proof}

\begin{cor}\label{cor:tt4con}
Every $4$-connected planar graph has a Tutte tree.
\end{cor}

\begin{proof}
If $T$ is a spanning tree that results of removing an edge from a
Hamiltonian cycle of $G$, then $T$ is a Tutte tree. Since Tutte~\cite{Tu56} proved
that every $4$-connected planar graph is Hamiltonian, the corollary follows.
\end{proof}

One could try to prove that every $3$-connected planar graph which has a
Tutte tree with three leaves is Hamiltonian. However this is not
true. Consider for instance the Herschel graph
(Figure~\ref{fig:herschel1}), which is the smallest nonhamiltonian
\defi{polyhedral} (planar $3$-connected) graph, and replace two degree
three vertices by a triangle as in Figure~\ref{fig:herschel}. The
resulting graph is also planar and nonhamiltonian but has a Tutte tree
with three leaves.

\begin{figure}[h!]
\centering
\includegraphics[width=0.4\textwidth]{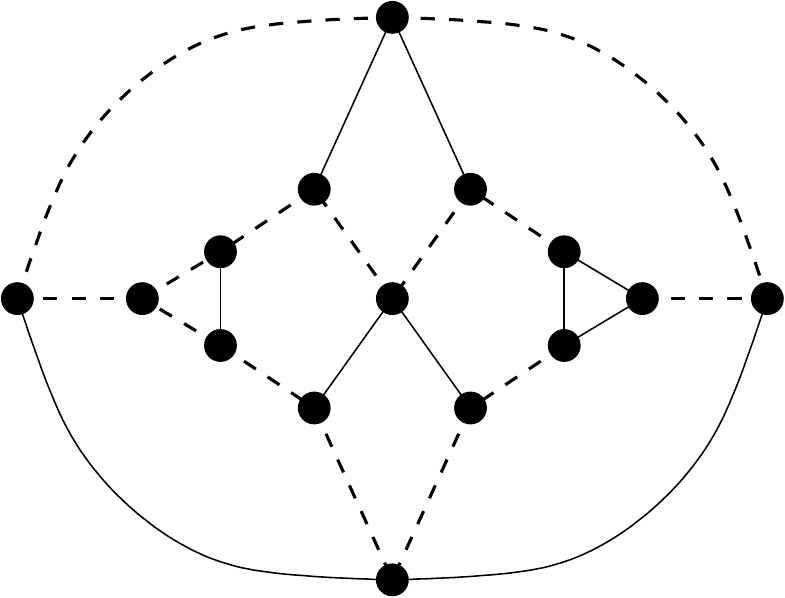}
\caption{Nonhamiltonian planar graph with a Tutte tree in dashed edges.}
\label{fig:herschel}
\end{figure}

\begin{lem}\label{lem:nontutte}
There are infinitely many planar $3$-connected graphs with no Tutte tree.
\end{lem}

\begin{proof}
Let $n \ge 5$ be an integer and $H_n$ be a $3$-connected planar
triangulation with $n$ vertices. 
We colour these vertices black.
For each face of $H_n$, we add a new vertex which is joined to the 
three black vertices in the face. 
We colour the new vertices white.
Let $G_n$ be the resulting graph.

Trivially $G_n$ is planar and $3$-connected. 
Since $G_n$ has $n$ black vertices then, from Euler's formula, 
it follows that $G_n$ has $2n-4$ white vertices. 
Since white vertices are not adjacent to each other, 
every spanning tree in $G_n$ has at least two white 
vertices as leaves. 
Hence, by Theorem~\ref{thm:tuttetree}, graph $G_n$ has no Tutte tree.
\end{proof}

For describing the structure of a $2$-connected graph with a
Tutte tree, we need to recall some concepts.  
Let $H$ be a proper subgraph of a connected graph $G$.  
We say $B$ is an \defi{$H$-bridge} in $G$ if:
\begin{enumerate}[(i)]
\item\label{bridge1} $B$ is  an edge of 
$E(G) \setminus E(H)$ with both ends in 
$V(H)$; or
\item\label{bridge2} $B$ is the union of a component 
$C$ of $G - H$ plus all edges connecting $C$ to $H$.
\end{enumerate}
The vertices of $H$ that have neighbors in $B$ are 
\defi{vertices of attachment} of $B$.

It follows from the definition that any two vertices of $B$ are
connected in $B$ by a path internally disjoint from $H$; and any two
different bridges intersect only in vertices of $H$.  For an
$H$-bridge $B$, the vertices $V(B) \setminus V(H)$ are its
\defi{internal vertices}. A bridge is \defi{trivial} if it has no
internal vertex; that is, it is an edge.

\begin{lem}\label{lem:tt2c}
Let $G$ be a $2$-connected graph and $\{u,v\}$ be a $2$-vertex cut. 
If $G$ has a Tutte tree, then exactly one of the following holds: 
\begin{inparaenum}[(i)]
\item there are two $\{u,v\}$-bridges in $G$ and at least one has
      a Hamiltonian $uv$-path; or
\item there are three $\{u,v\}$-bridges in $G$, one is trivial
      and another one has a Hamiltonian $uv$-path.
\end{inparaenum}
\end{lem}

\begin{proof}
Let $\{u,v\}$ be a $2$-vertex cut and suppose that 
$G$ has a Tutte tree $T$.
Let $uTv$ be the unique path between $u$ and $v$ in $T$.
By the definition of bridge, $uTv$ is  contained in a $\{u,v\}$-bridges. 
Since we only consider simple graphs, at most one 
$\{u,v\}$-bridge is trivial. 
If there are two $\{u,v\}$-bridges, then $uTv$ must be a 
Hamiltonian $uv$-path in its bridge, otherwise $uTv$ is separating. 
If there are three $\{u,v\}$-bridges, then one must be trivial
and $uTv$ must be a Hamiltonian $uv$-path in another bridge that contains $uTv$,
otherwise $uTv$ would be separating. 
In the case of at least four $\{u,v\}$-bridges, then $G-uTv$
has two components. 
\end{proof}

The structure of a $3$-connected graph with a Tutte tree is very related to Hamiltonicity, 
as well as graphs with a $2$-vertex cut as we observed in the previous result, 
so the problem of finding a Tutte tree turns out to be a hard problem.

\begin{lem}\label{lem:ttnp}
Deciding if a cubic planar $3$-connected graph with no facial triangles 
has a Tutte tree is an NP-complete problem.
\end{lem}

\begin{proof}
Let $G$ be a cubic planar $3$-connected graph with no facial triangles. 
If $G$ has a triangle $\Delta$, then $G-\Delta$ is disconnected.  
It follows that it is not possible to get a spanning tree of $G$ whose leaves are
exactly the vertices of $\Delta$.
Thus, from Theorem~\ref{thm:tuttetree}, 
$G$ has a Tutte tree if and only if it is Hamiltonian. 
On other hand,
Garey, Johnson, and Tarjan~\cite{GaJo76} proved that, 
for cubic planar $3$-connected graphs,
deciding the existence of a Hamiltonian cycle is an NP-complete problem. 
Hence the lemma follows.
\end{proof}

\begin{proof}[Proof of Theorem~\ref{thm:3contt}]
Let $P$ be a path in $T$. 
Clearly $P$ contains at most two vertices of $L(T)$, 
and therefore $G[L(T)]-P$ is connected. Note that every
component of $G-P$ has a leaf of $T$. 
It follows that $P$ is nonseparating and hence $T$ is a Tutte tree.
\end{proof}


In the nonplanar case, the $3$-connectedness in $L(T)$ is a sufficient
but not necessary condition. There exist Tutte trees where
$G[L(T)]$ is for instance an independent set (Figure~\ref{fig:k5proj}), a path (Figure~\ref{fig:k3,3}), or a cycle (Figure~\ref{fig:petersen}).

\begin{figure}[h!]
        \centering
        \begin{subfigure}[t]{0.3\textwidth}
                \includegraphics[width=\textwidth]{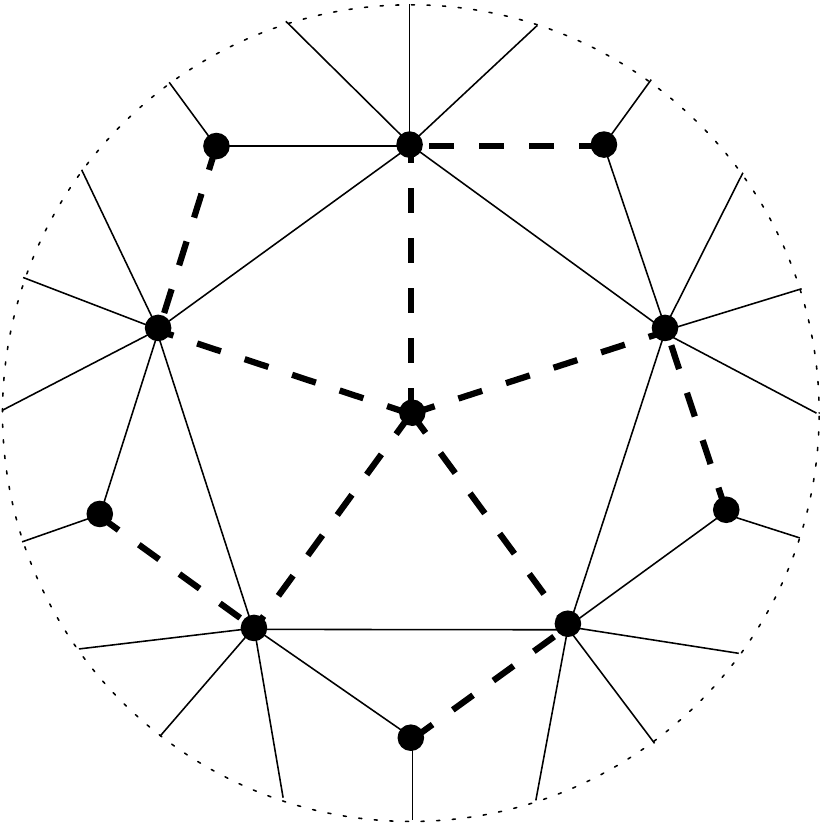}
                \caption{Barycentric subdivision of $K_5$ in the projective plane.}
                \label{fig:k5proj}
        \end{subfigure}%
        ~ 
        \begin{subfigure}[t]{0.3\textwidth}
                \includegraphics[width=\textwidth]{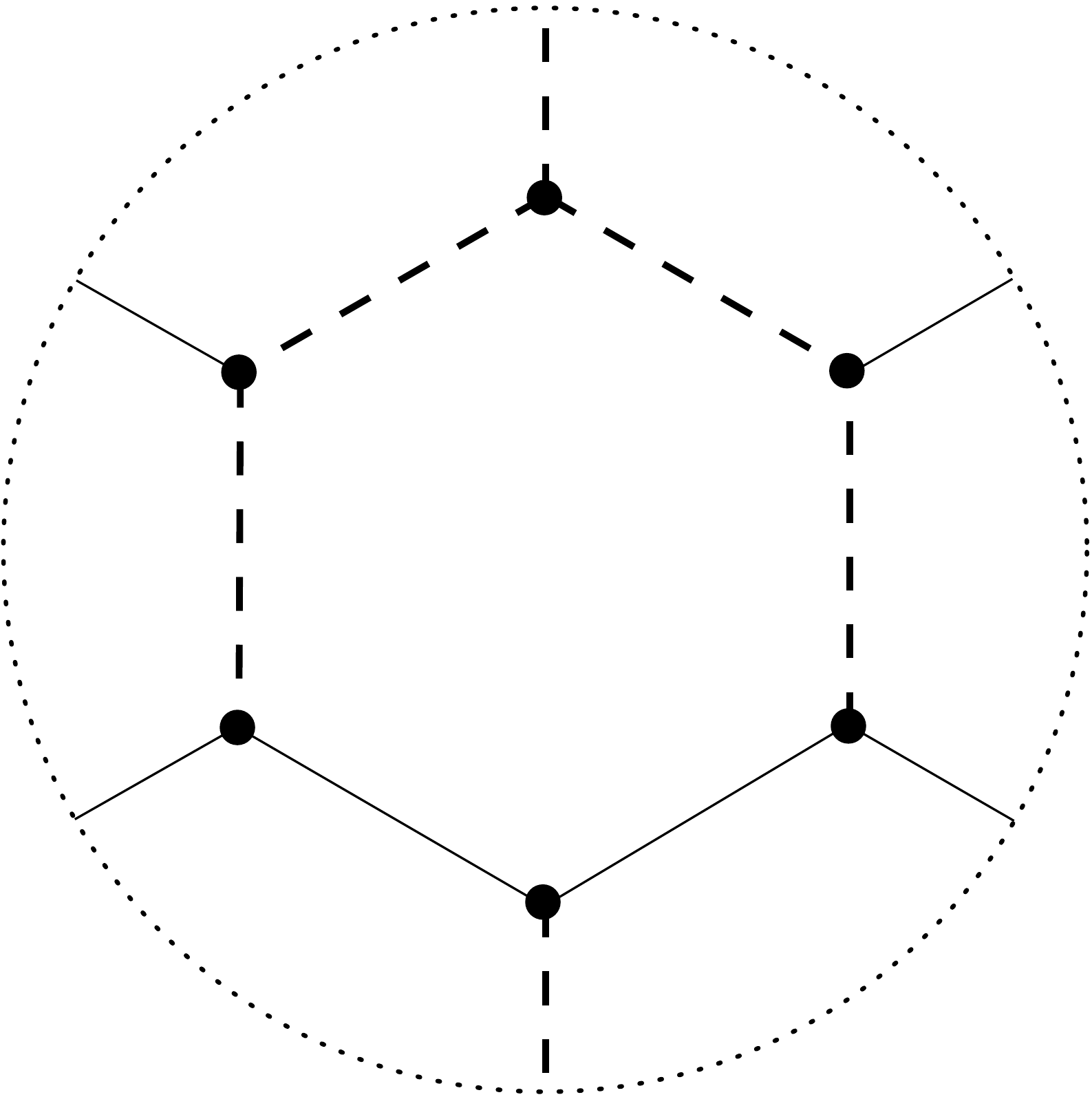}
                \caption{$K_{3,3}$ in the projective plane.}
                \label{fig:k3,3}
        \end{subfigure}
        ~ 
        \begin{subfigure}[t]{0.3\textwidth}
                \includegraphics[width=\textwidth]{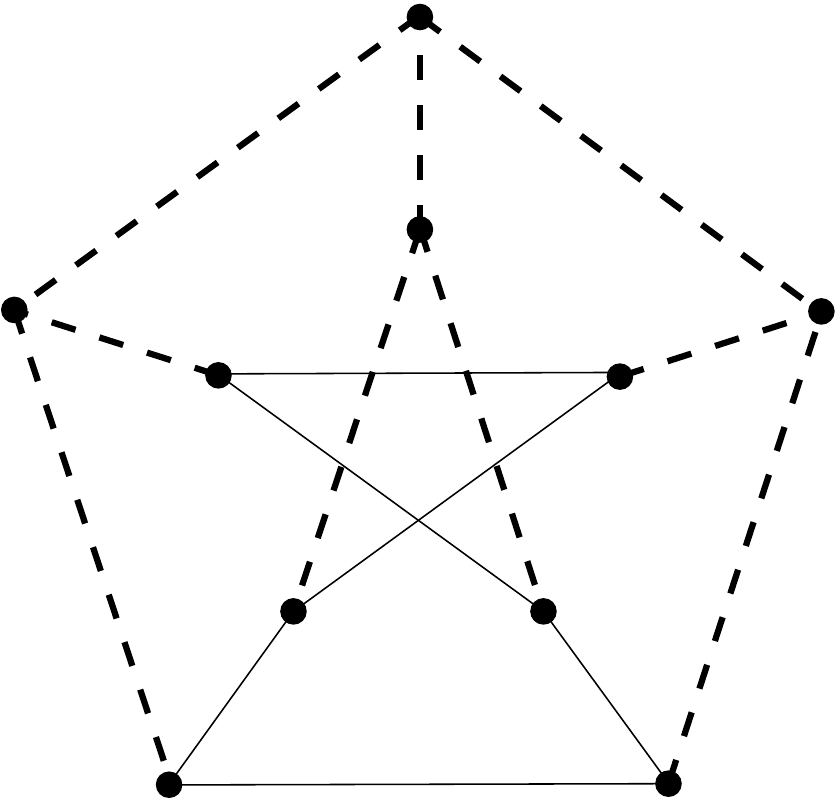}
                \caption{Petersen graph.}
                \label{fig:petersen}
        \end{subfigure}
        \caption{$3$-connected nonplanar graphs with a Tutte tree in dashed 
          edges.}\label{fig:nonplanartt}
\end{figure}

Unlike the planar graphs, the existence of a Tutte tree in a nonplanar
graph implies neither a Hamiltonian cycle nor a Hamiltonian path. See
for instance the $3$-connected nonplanar graph $K_{3,5}$
(Figure~\ref{fig:k3,5}).

\begin{figure}[h!]
\centering
\includegraphics[width=0.3\textwidth]{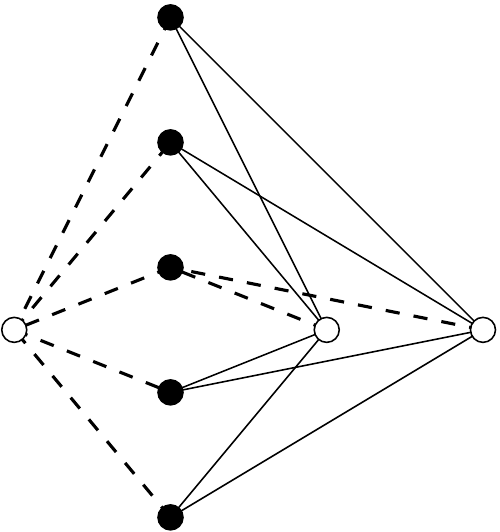}
\caption{Nontraceable and thus nonhamiltonian,\\ $3$-connected nonplanar graph with a Tutte tree in dashed edges.}
\label{fig:k3,5}
\end{figure}

As in the planar case, $3$-connectedness is not a sufficient condition for the existence of a Tutte tree in a nonplanar graph. 
\begin{figure}[!ht]
\centering
\includegraphics[width=0.3\textwidth]{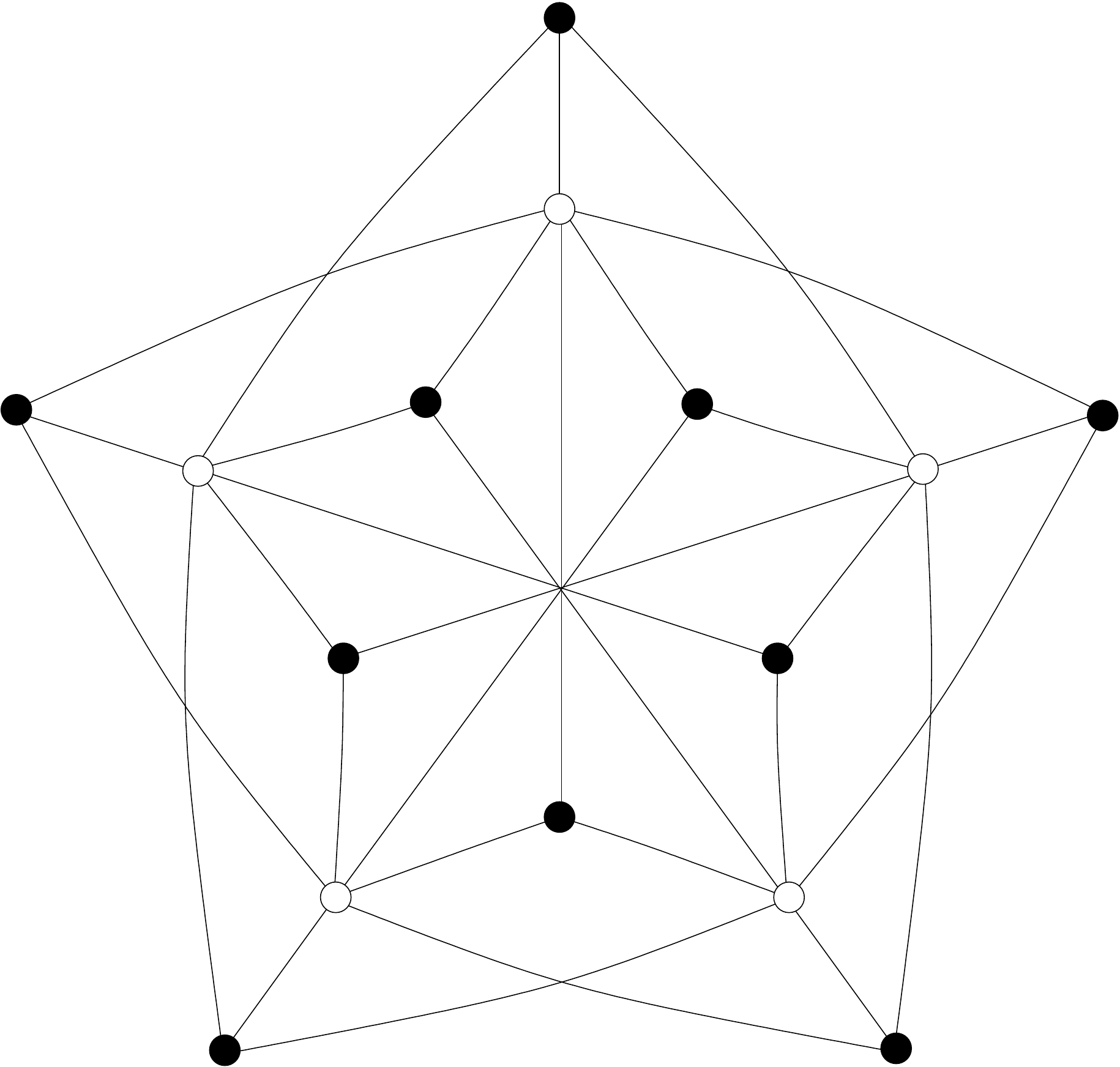}
\caption{Graph $S$, which is $3$-connected nonplanar\\ 
          with no Tutte tree.}
\label{fig:stark5}
\end{figure}

\begin{prop}\label{prop:star}
Let $S$ be the graph obtained as follows: take a set of
five white vertices, then, for each three of these, add a new black
vertex and join it to each of the three vertices
(Figure~\ref{fig:stark5})  
Then $S$ is 3-connected and does not have a Tutte tree.
\end{prop}

\begin{proof}
Clearly $S$ is 3-connected and bipartite. 
Let us suppose that $G$ has a Tutte tree $T$.

If $T$ has a path of length six between two white
vertices, i.e., a path with four white vertices, then this is a
separating path since at least one of the black vertices associated
with a triple formed by these four white vertices is not an inner
vertex in such a path. 
So, we may assume $T$  does not contains 
a path with four white vertices.

Since no black vertex is adjacent
to all five white vertices, $T$ has a path with at least three white
vertices. Let $P$ be a path in $T$ with exactly three
white vertices, two of which are ends.
Let $v_1,v_2$ denote the remaining two white vertex of $S - P$.  Then,
the path connecting $v_i$ to $P$ in $T$, called $v_iP$-path, has as vertex of
attachment an inner vertex of $P$
or there would be a path with four white vertices in $T$.  
Let $T'$ be the subtree resulting
from the union of $P$ and the $v_1P$-,$v_2P$-paths in $T$.
One can verify that, in all possible configurations of $T'$, 
the number of paths with three white vertices is greater 
than the number of black vertices in $T'$ and therefore 
$T$ has a separating path, 
namely, the path containing three white vertices in $T'$
whose corresponding black vertex is not in $T'$.
\end{proof}


\section{Fundamental Tutte trees}
\label{sec:fundtuttetree}

Unlike a Tutte tree, if a graph is connected but not 2-connected, a
fundamental Tutte tree can exist.  
For instance, a path has a fundamental Tutte tree: itself. 
Indeed, it does not have any (fundamental) cycle. 
Note that every Tutte tree is a fundamental Tutte tree.

In order to better understand graphs with a fundamental Tutte tree, we
explore their structure.
\begin{lem}\label{lem:ftt1c}
A graph $G$ with a cut vertex has a fundamental Tutte tree if and only if
\begin{inparaenum}[(i)]
\item every nontrivial block of $G$ is a leaf in the block tree of $G$,
\item every nontrivial block of $G$ is Hamiltonian, and
\item every articulation vertex of a nontrivial block of $G$ has degree two in the nontrivial block.
\end{inparaenum}
\end{lem}

\begin{proof}
\noindent{($\Rightarrow$)} 
\emph{(i)}
 If a nontrivial block is not a leaf in the block tree, then it has
 at least two articulation vertices, say $u$ and $v$. 
 Since it is $2$-connected, it contains a cycle with $u$ and $v$.  
 It is known that the fundamental cycles form a basis of the cycle
 space~\cite[Theorem~1.9.5]{Di:book}. 
 It follows that any spanning tree of $G$ has
 a fundamental cycle containing $u$ inside the block. 
 Such a cycle separates the blocks attached to $u$ from the 
 blocks attached to $v$. 
 Thus every nontrivial block must be a leaf in the block tree of $G$.

\emph{(ii)} 
Let $u$ be an articulation point of a nontrivial block of $G$. 
Then $u$ is contained in some fundamental cycle on the tree in this block.
If such a cycle does not contain all vertices in the block, then it
separates the blocks attached to $u$ from the rest of the
block. 
Hence this fundamental cycle containing $u$ also contains all
vertices in the block. 
It follows that the block is Hamiltonian and the fundamental 
Tutte tree restricted to the block is a Hamiltonian path of the block.

\emph{(iii)} 
Let $u$ be as in the proof of \emph{(ii)}. 
Suppose that $u$ has degree greater than two in the nontrivial block. 
Then at least one edge adjacent to~$u$ is a chord of the fundamental 
cycle of the tree containing~$u$ in the block, 
which is also a Hamiltonian cycle in the block, as argued for \emph{(ii)}. Thus the
fundamental cycle induced by the chord does not contain all vertices
of the block and separates the blocks attached to~$u$ from the
rest of the block. Hence $u$ has degree two in the block.

\noindent{($\Leftarrow$)} 
Let $T'$ be obtained from the block tree of $G$ by 
removing the nontrivial blocks, which are leaves.
For every nontrivial block, we choose a Hamiltonian path that results
of the Hamiltonian cycle after removing an edge incident to the
articulation point. 
Joining each Hamiltonian path to the tree $T'$ yields a spanning tree $T$. 
It is not hard to see that $T$ is a fundamental Tutte tree of $G$.
\end{proof}


We now start to describe the structure of graphs with a 2-vertex cut
that have a Tutte tree.


\begin{lem}\label{lem:ftt2c}
Let $G$ be a $2$-connected graph and $\{u,v\}$ be a $2$-vertex cut. 
If $G$ has a fundamental Tutte tree, then there are at most 
three $\{u,v\}$-bridges in $G$.
\end{lem}

\begin{proof}
Suppose that $\{u, v\}$ has at least four $\{u,v\}$-bridges and 
that $G$ has a fundamental Tutte tree $T$. 
Because of the $2$-connectedness, every $\{u,v\}$-bridge has
a~$uv$-path. 
It follows that at least three $\{u,v\}$-bridges have an
edge not in $T$ that induces a fundamental cycle, 
and every fundamental cycle is contained in the union of at most 
two $\{u,v\}$-bridges. 
Since we only consider simple graphs, at most one $\{u,v\}$-bridge is trivial. 
If $G$ has a trivial $\{u,v\}$-bridge then it
belongs to $T$, otherwise its fundamental cycle separates the others
(nontrivial) $\{u,v\}$-components. 
If a fundamental cycle is contained in the union of 
two $\{u,v\}$-bridges, then it  separates the 
others $\{u,v\}$-components. 
(Note that whenever $G$ has a trivial $\{u,v\}$-bridge, 
every fundamental cycle in two $\{u,v\}$-bridges contains 
it and the remaining two $\{u,v\}$-bridges are nontrivial.) 
Thus, every fundamental cycle of $T$ 
is completely contained in a $\{u,v\}$-bridge.
Therefore, the fundamental cycles of $T$ do not generate the
cycles containing edges of two $\{u,v\}$-bridges. 
This contradicts the fact that the fundamental cycles 
of a spanning tree
form a basis for the cycle space~\cite[Theorem~1.9.5]{Di:book}.
Hence such a tree $T$ cannot exist.
\end{proof}

\begin{cor}\label{cor:uvhp}
Let $G$ be a $2$-connected graph and $\{u,v\}$ be a $2$-vertex cut. 
If $G$ has a fundamental Tutte tree, 
then every $\{u,v\}$-bridge has a Hamiltonian $uv$-path,
except possibly one. 
Moreover, if there are exactly three $\{u,v\}$-bridges, then every $\{u,v\}$-bridge has a Hamiltonian $uv$-path.
\end{cor}

\begin{proof}
Let $T$ be a fundamental Tutte tree in $G$. 
Since the fundamental cycles of $T$ induce a basis of the cycle space, 
there exists at least one fundamental cycle  
contained in the union of two $\{u,v\}$-bridges, $B_1$ and $B_2$, say $C_{12}$.
If there are exactly two $\{u,v\}$-bridges, then $C_{12}$ must contain all 
vertices of at least one of $B_1$ or $B_2$, 
and the corollary follows. 
If there are exactly three $\{u,v\}$-bridges, let $B_3$ be the other one.
Similarly, there exist fundamental cycles $C_{ij}$ contained in the union of $B_i$ and $B_j$. 
It follows that $\{C_{ij}\}$  induces a Hamiltonian $uv$-path in each $\{u,v\}$-bridge.
\end{proof}

A \defi{series extension} of a graph is the subdivision of an edge; a
\defi{parallel extension} is the addition of a new edge joining two
adjacent vertices. We momentarily consider multigraphs.
A graph is \defi{series-parallel} if it can be obtained from $K_2$ by a sequence
of series and parallel extensions.  
It is well-known that a graph is series-parallel if and only if it has no $K_4$
minor~\cite[Section~11.2]{BrLe99}. 

\begin{thm}\label{thm:spg}
A $2$-connected series-parallel graph $G$ has a fundamental 
Tutte tree if and only if, for every $2$-vertex cut $\{u,v\}$,
\begin{inparaenum}[(i)]
\item\label{itm:1spg} there are at most three $\{u,v\}$-bridges;
\item\label{itm:2spg} if there are exactly two $\{u,v\}$-bridges, 
one has a  Hamiltonian $uv$-path; 
\item\label{itm:3spg} 
if there are three $\{u,v\}$-bridges, 
each one has a Hamiltonian $uv$-path and at most one 
is $2$-edge-connected.
\end{inparaenum}
\end{thm}

\begin{proof}
\noindent{($\Rightarrow$)}
Recall that a series-parallel graph is planar and not $3$-connected. 
If $G$ is a triangle, then it has a fundamental Tutte 
tree and there is no $2$-vertex cut. 
Otherwise, there exists at least one $2$-vertex cut $\{u,v\}$ in $G$. 
Suppose $G$ has a fundamental Tutte tree $T$.
Items~(\ref{itm:1spg}) and~(\ref{itm:2spg}) 
follow from Lemma~\ref{lem:ftt2c} and Corollary~\ref{cor:uvhp}, respectively.  

Now, suppose that there are three 
$\{u,v\}$-bridges and two of them are $2$-edge-connected. 
It follows from Corollary~\ref{cor:uvhp} that
every $\{u,v\}$-bridge has a Hamiltonian $uv$-path. 
At least one of the $2$-edge-connected $\{u,v\}$-bridges, 
say $B$, does not have a $uv$-path in $T$.
Then $B$ has an edge $xy$ whose fundamental cycle
$C_{xy}$ contains an edge of another $\{u,v\}$-bridge $B'$.
If $C_{xy}$ is not a Hamiltonian cycle in the subgraph induced 
by the union of $B$ and $B'$, then $C_{xy}$ is separating. 
It follows that $C_{xy}$ restricted to $B$ 
is a Hamiltonian $uv$-path.
Without loss of generality, we may assume that there exist 
a $ux$-path and a $yv$-path in $T$ 
whose union contains all vertices of $B$. 
Since $xy$ is not a cut edge, there exists an edge $x'y'$ not in $T$ 
connecting the $ux$-path and the $yv$-path.
The fundamental cycle $C_{x'y'}$ is separating, which is a
contradiction. 
Then at most one $\{u,v\}$-bridge is $2$-edge-connected, 
which proves item~(\ref{itm:3spg}).

\noindent{($\Leftarrow$)} 
Item~(\ref{itm:1spg}) is a necessary
condition for the existence of a fundamental Tutte tree.  

We define a fundamental Tutte tree according to the number of bridges:
\begin{description}
\item 1. \textit{For every $2$-vertex cut $\{u,v\}$ there are exactly two 
$\{u,v\}$-bridges.}
\end{description}

\noindent
Let $C$ be a longest cycle in $G$. 
If $C$ is a Hamiltonian cycle, we are done. 
So suppose that there is a vertex $w$ not in $C$. 
Since $G$ is $2$-connected, there are two distinct vertices $x$ and
$y$ in $C$, and a $wx$-path and a $wy$-path internally disjoint from
each other.
If $x$ and $y$ are adjacent in $C$, 
then we get a longer cycle replacing the edge
$xy$ in $C$ by the concatenation at $w$ of the two paths. 
It follows that there is at least one vertex in each 
$xy$-path of $C$.
Because there are exactly two $\{x,y\}$-bridges, $G$ has 
a $K_4$ minor.
Hence $C$ is a Hamiltonian cycle. 
In this case, a fundamental Tutte tree is obtained from a Hamiltonian
cycle by removing one edge.

\begin{description}
\item 2. \textit{There exists a $2$-vertex cut $\{u,v\}$ with three 
$\{u,v\}$-bridges.}
\end{description}

\noindent
By item~(\ref{itm:3spg}), every $\{u,v\}$-bridge
has a Hamiltonian $uv$-path and at most one of the $\{u,v\}$-bridges  is
$2$-edge-connected. 
Let $xy$ and $x'y'$ be cut edges, one for each no $2$-edge-connected
$\{u,v\}$-bridge. 
Note that, for the respective $\{u,v\}$-bridge, every 
Hamiltonian $uv$-path must contain the cut edge.  
The fundamental Tutte tree is as follows. 
Take a  Hamiltonian $uv$-path in each
$\{u,v\}$-bridge and remove edges $xy$ and $x'y'$. 
It is easy to verify that every fundamental cycle is nonseparating.
\end{proof}

\begin{proof}[Proof of Theorem~\ref{thm:fundtuttetree}]
The case when $G$ has a cut vertex follows from
Lemma~\ref{lem:ftt1c}. So we may assume that $G$ is at least
$2$-connected. Let $T$ be a fundamental Tutte tree and suppose that
there exist a plane drawing $D$ of $G$ without all the leaves in the
same face. Let $u$ and $v$ be two leaves in different faces of
$D$. Then there exist a cycle $C$ that separates $u$ from
$v$~\cite[Theorem~3.1]{Tu75}; 
that is, a cycle containing neither
$u$ nor $v$ and such that they are in distinct regions of the Jordan
curve $D[C]$. Since the fundamental cycles form a basis of the cycle
space~\cite[Theorem~1.9.5]{Di:book}, 
it follows that there exists a
fundamental cycle $C_e$ of $T$ that separates~$u$ from $v$.  Because
$G$ is planar, $C_e$ is a separating fundamental cycle, and so we have
a contradiction.
\end{proof}

\begin{cor}
There are infinitely many planar $3$-connected graphs with no
fundamental Tutte tree.
\end{cor}

\begin{proof}
Let $n \ge 5$ be an integer and $G_n$ be as defined in the proof of Lemma~\ref{lem:nontutte}. It follows that every spanning tree of $G_n$ has at least two leaves in different faces and therefore $G_n$ has no fundamental Tutte tree.
\end{proof}

If a graph has a Tutte tree, then such tree is also a fundamental
Tutte tree. However, there exist (planar) $3$-connected graphs with a
fundamental Tutte tree, but without a Tutte tree, as for instance the
Herschel graph (Figure~\ref{fig:herschel1}), which is nonhamiltonian
and has no triangles.

\begin{figure}[h!]
\centering
\includegraphics[width=0.4\textwidth]{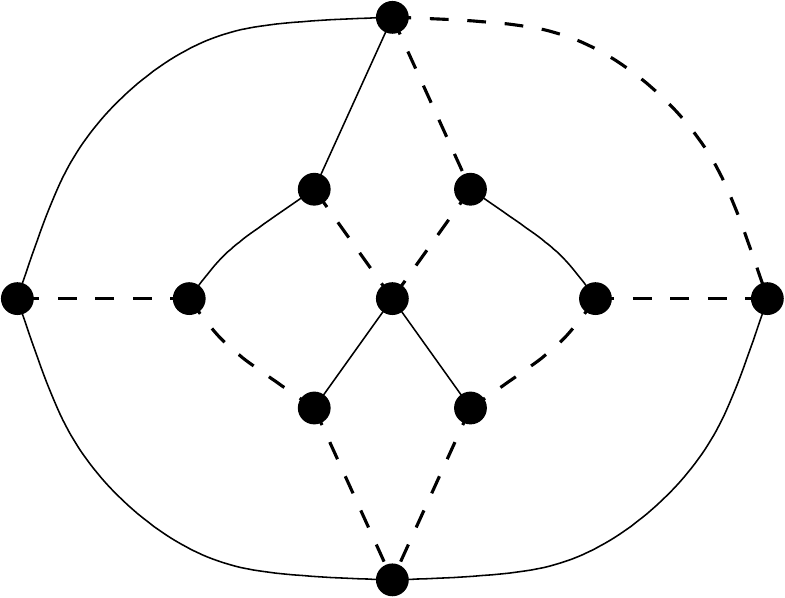}
\caption{Herschel graph, with a fundamental Tutte tree in dashed edges, 
       has no Tutte tree.}
\label{fig:herschel1}
\end{figure}

In the case of planar graphs, traceability is a necessary condition
(Corollary~\ref{cor:tuttehamil})
but not sufficient (Herschel graph) 
for the existence of a Tutte tree. 
On the other
hand, there are nontraceable (and thus nonhamiltonian) cubic planar
$3$-connected graphs~\cite{Za80} with a fundamental Tutte tree,
e.g., the Zamfirescu graph (Figure~\ref{fig:zamfi}).

\begin{figure}[h]
\centering
\resizebox{0.5\textwidth}{!}{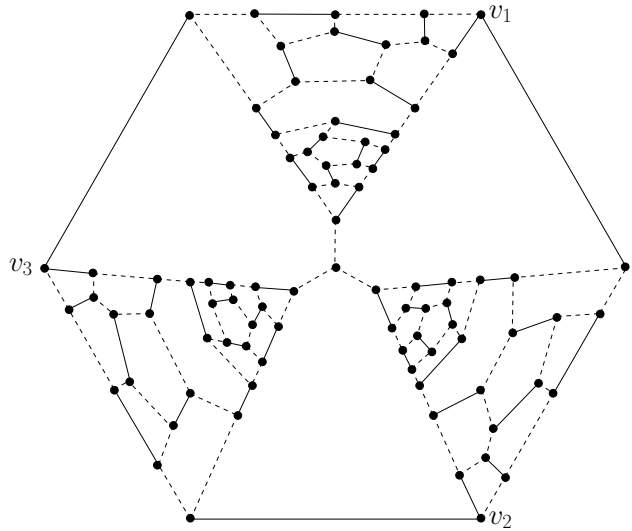}
\caption{Zamfirescu graph which is nontraceable cubic planar $3$-connected\\ 
with a fundamental Tutte tree in dashed edges.}
\label{fig:zamfi}
\end{figure}

Having a spanning tree with all leaves in a same face does not guarantee
that a plane graph has a fundamental Tutte tree. 

\begin{prop}\label{prop:noftt}
Let $G$ be the planar graph obtained as follows: 
take the Zamfirescu graph 
and let $v_1$ and~$v_2$ be the vertices displayed in Figure~\ref{fig:zamfi}.
Let $w$ be a new vertex and joint it to $v_1$ and $v_2$.
Finally add the edge $v_1v_2$.
Then: 
\begin{inparaenum}[(i)]
\item \label{item:noftt1}
there exist a spanning tree $T$ of $G$ and a plane drawing of $G$ 
for which all leaves of $T$ are in the same face; and
\item \label{item:noftt2}
$G$ has no fundamental Tutte tree.
\end{inparaenum}
\end{prop}

\begin{proof}
We get a spanning tree $T$ of $G$ just adding the edge $v_1w$ to the
spanning tree in dashed edges shown in Figure~\ref{fig:zamfi}. It is
easy to obtain a plane drawing of $G$ such that all leaves of $T$ are in
the same face. This proves~\eqref{item:noftt1}.

By Corollary~\ref{cor:uvhp}, if $G$ has a fundamental Tutte tree then
every $\{v_1, v_2\}$-bridge should have a  Hamiltonian
$v_1v_2$-path. The Zamfirescu graph is a $\{v_1,v_2\}$-bridge of $G$, 
but has no a Hamiltonian $v_1v_2$-path. 
It follows that $G$ has no fundamental Tutte tree.
\end{proof}


\section{Concluding Remarks}\label{sec:cr}

We have not been able to give a necessary and sufficient condition
for the existence of either a Tutte tree or a fundamental Tutte tree,
so some interesting questions remain open.

Every $4$-connected planar graph has a (fundamental) Tutte tree, since
it is Hamiltonian. 
We do not know a $4$-connected nonplanar
graph with no (fundamental) Tutte tree. 
This yields the following question.

\begin{ques}
Does every $4$-connected graph have a (fundamental) Tutte tree?
\end{ques}

Proposition~\ref{prop:noftt} shows that the existence of a spanning
tree with all its leaves in the same face is not a sufficient
condition for having a fundamental Tutte tree.  
Nevertheless the graph
constructed is not $3$-connected. So some questions arise.

\begin{ques}
Is there a $3$-connected planar graph $G$ with a spanning tree with 
all its leaves in the same face, such that $G$ has not a 
fundamental Tutte tree?
\end{ques}

\begin{ques}
Let $G$ be a $3$-connected planar graph with a spanning tree
 whose all  leaves are in the same face. 
Does $G$ have a fundamental Tutte tree?
\end{ques}





\bibliographystyle{elsarticle-num}
\bibliography{bibtuttetree}


\end{document}